\documentclass[12pt]{amsart}
\usepackage{amssymb}
\usepackage{enumerate}
\usepackage{graphicx}
\usepackage[all,cmtip]{xy}
 \usepackage{tikz-cd}
 \usepackage{orcidlink}
\makeatletter
\@namedef{subjclassname@2010}{%
  \textup{2010} Mathematics Subject Classification}
\makeatother
\newtheorem{thm}{Theorem}[section]
\newtheorem{corollary}[thm]{Corollary}
\newtheorem{lemma}[thm]{Lemma}
\newtheorem{proposition}[thm]{Proposition}
\theoremstyle{definition}
\newtheorem{definition}[thm]{Definition}
\newtheorem{remark}[thm]{Remark}
\newtheorem{example}[thm]{Example}
\begin{document}
\baselineskip=15pt
\title{Uniformly $S$-pseudo-projective modules}

\author[M. Adarbeh]{Mohammad Adarbeh $^{(\star)}$ \orcidlink{0000-0003-4702-3659}}
\address{Department of Mathematics, Birzeit University, Birzeit,  Palestine}
\email{madarbeh@birzeit.edu}
\author[M. Saleh]{Mohammad Saleh  \orcidlink{0000-0002-4254-2540}}
\address{Department of Mathematics, Birzeit University, Birzeit,  Palestine}
\email{msaleh@birzeit.edu}

\thanks{$^{(\star)}$ Corresponding author}
\date{}

\begin{abstract}
In this paper, we introduce the notion of uniformly $S$-pseudo-projective ($u$-$S$-pseudo-projective) modules as a generalization of $u$-$S$-projective modules. Let $R$ be a ring and $S$ a multiplicative subset of $R$. An $R$-module $P$ is said to be $u$-$S$-pseudo-projective if for any submodule $K$ of $P$, there is $s\in S$ such that for any $u$-$S$-epimorphism $f:P\to \frac{P}{K}$, $sf$ can be lifted to an endomorphism $g:P\to P$. We prove that an $R$-module $M$ is $u$-$S$-quasi-projective if and only if $M\oplus M$ is $u$-$S$-pseudo-projective. Also, we prove that if $A\oplus B$ is $u$-$S$-pseudo-projective, then any $u$-$S$-epimorphism $f:A\to B$ $u$-$S$-splits. We give characterizations of certain classes of rings, such as $u$-$S$-semisimple and strongly $S$-perfect rings. 
\end{abstract}



\maketitle

\section{Introduction}
In this paper, all rings are commutative with a nonzero identity, and all modules are unitary unless otherwise stated. Recall that a nonempty subset $S$ of a ring $R$ is called a multiplicative subset of $R$ if $1 \in S$, $0 \notin S$, and $s_1s_2\in S$ for all $s_1,s_2 \in S$. Let $S$ be a multiplicative subset of a ring $R$. Recall from \cite{Z} that an $R$-module $U$ is called a $u$-$S$-torsion module if $sU = 0$ for some $s \in S$. Now, we recall the following basic notions: Let $A$, $B$, and $C$ be $R$-modules.
\begin{enumerate}
    \item[(a)] An $R$-homomorphism $f: A \to B$ is called a $u$-$S$-monomorphism ($u$-$S$-epimorphism) if $\text{Ker}(f)$ ($\text{Coker}(f)$) is a $u$-$S$-torsion module. An $R$-homomorphism $f: A \to B$ is called a $u$-$S$-isomorphism if $f$ is both a $u$-$S$-monomorphism and a $u$-$S$-epimorphism \cite{Z}.
    \item[(b)]  An $R$-sequence $A \xrightarrow{f} B \xrightarrow{g} C$ is said to be $u$-$S$-exact if there exists $s \in S$ such that $s\text{Ker}(g) \subseteq \text{Im}(f)$ and $s\text{Im}(f) \subseteq \text{Ker}(g)$. A $u$-$S$-exact sequence $0 \to A \to B \to C \to 0$ is called a short $u$-$S$-exact sequence \cite{ZQ}. 
\item[(c)] A short $u$-$S$-exact sequence $0\to A \xrightarrow{f} B \xrightarrow{g} C\to 0$ is said to be $u$-$S$-split (with respect to $s$) if there is $s\in S$ and an $R$-homomorphism $f':B \to A$ such that $f'f=s1_{A}$, where $1_{A}:A\to A$ is the identity map on $A$ \cite{ZQ}.
 \end{enumerate} M. Z. Chen et al. \cite{QK} introduced the notion of $u$-$S$-injective modules. An $R$-module $E$ is called $u$-$S$-injective if for any $u$-$S$-exact sequence $0 \to
A\to B\to C \to 0$, the induced sequence
 \begin{center}
     $0 \to \text{Hom}_{R}(C,E)\to \text{Hom}_R(B,E)\to \text{Hom}_R(A,E) \to 0$
 \end{center}
 is $u$-$S$-exact. Dually, X. Zhang and W. Qi \cite{ZQ} introduced the notion of $u$-$S$-projective modules. An $R$-module $P$ is called $u$-$S$-projective if for any $u$-$S$-exact sequence $0 \to
A\to B\to C \to 0$, the induced sequence
 \begin{center}
$0 \to \text{Hom}_{R}(P,A)\to \text{Hom}_R(P,B)\to \text{Hom}_R(P,C) \to 0$     
 \end{center}
 is $u$-$S$-exact. They also introduced the notions of $u$-$S$-semisimple modules and $u$-$S$-semisimple rings. An $R$-module $M$ is called $u$-$S$-semisimple if any $u$-$S$-short exact sequence $0\to A\to M\to C\to 0$ is $u$-$S$-split. A ring $R$ is called $u$-$S$-semisimple if any free $R$-module is $u$-$S$-semisimple. Recently, M. Adarbeh and M. Saleh \cite{MM} introduced the notion of $u$-$S$-injective ($u$-$S$-projective) relative to a fixed module. An $R$-module $M$ is called $u$-$S$-injective ($u$-$S$-projective) relative to a fixed module $N$ if for any $u$-$S$-monomorphism $f:K\to N$ ($u$-$S$-epimorphism $g:N\to L$), the induced map $\text{Hom}_{R}(f,M): \text{Hom}_R(N,M)\to \text{Hom}_R(K,M)$ ($\text{Hom}_{R}(M,g): \text{Hom}_R(M,N)\to \text{Hom}_R(M,L)$) is a $u$-$S$-epimorphism. They also introduced the notion of $u$-$S$-quasi-injective ($u$-$S$-quasi-projective) modules. An $R$-module $M$ is called $u$-$S$-quasi-injective ($u$-$S$-quasi-projective) if it is $u$-$S$-injective ($u$-$S$-projective) relative to itself. Afterwards, M. Adarbeh and M. Saleh \cite{MM2} introduced the notion of $u$-$S$-pseudo-injective modules. They defined an $R$-module $E$ to be $u$-$S$-pseudo-injective if for any submodule $K$ of $E$, there is $s\in S$ such that for any $u$-$S$-monomorphism $f:K\to E$, $sf$ can be extended to an endomorphism $g:E\to E$. 
 
 The aim of this paper is to introduce and study the notion of $u$-$S$-pseudo-projective modules as a dual notion of $u$-$S$-pseudo-injective modules, as well as a uniformly $S$-version of pseudo-projective modules. An $R$-module $P$ is said to be $u$-$S$-pseudo-projective if for any submodule $K$ of $P$, there is $s\in S$ such that for any $u$-$S$-epimorphism $f:P\to \frac{P}{K}$, $sf$ can be lifted to an endomorphism $g:P\to P$. From \cite[Theorem 2.3]{MM}, we conclude that an $R$-module $P$ is $u$-$S$-quasi-projective if and only if for any submodule $K$ of $P$, there is $s\in S$ such that for any $R$-homomorphism $f:P\to \frac{P}{K}$, $sf$ can be lifted to an endomorphism $g:P\to P$. From \cite[Remark 3.2 (1)]{MM} and the last fact, we have the following implications: 
\begin{center}
$u$-$S$-projective $\Rightarrow$ $u$-$S$-quasi-projective $\Rightarrow$ $u$-$S$-pseudo-projective.
\end{center}

Section 2 is devoted to studying some properties of $u$-$S$-pseudo-projective modules. For example, Proposition \ref{prop1} (2) shows that if $A\oplus B$ is a $u$-$S$-pseudo-projective module, then so are $A$ and $B$. However, the converse of Proposition \ref{prop1} (2) need not be true (see Example \ref{ex2}). Corollary \ref{cor1} shows that an $R$-module $M$ is $u$-$S$-quasi-projective if and only if $M\oplus M$ is $u$-$S$-pseudo-projective. Proposition \ref{prop2} proves that if $A\oplus B$ is a $u$-$S$-pseudo-projective, then any $u$-$S$-epimorphism $f:A\to B$ $u$-$S$-splits. 

Section 3 deals with some characterizations of $u$-$S$-semisimple rings (Theorem \ref{thm1}), strongly $S$-perfect rings (Theorem \ref{thmp}), and rings in which every $u$-$S$-quasi-projective ($u$-$S$-pseudo-projective) module is $u$-$S$-projective Theorem \ref{thm2} (Theorem \ref{thm3}). 
\section{$u$-$S$-pseudo-projective modules}\label{d}
Recall that an $R$-module $P$ is called pseudo-projective if for any submodule $K$ of $P$, every epimorphism $f: P \to \frac{P}{K}$ can be lifted to an endomorphism $g : P \to P$ \cite{SS}. We introduce the uniformly $S$-version of pseudo-projective modules. 

\begin{definition}\label{def2}
   Let $S$ be a multiplicative subset of a ring $R$. An $R$-module $P$ is said to be $u$-$S$-pseudo-projective if for any submodule $K$ of $P$, there is $s\in S$ such that for any $u$-$S$-epimorphism $f:P\to \frac{P}{K}$, $sf$ can be lifted to an endomorphism $g:P\to P$.
   \end{definition}

\begin{remark} \label{rem1}
      Let $S$ be a multiplicative subset of a ring $R$ and $P$ an $R$-module. 
      \begin{enumerate}
          \item[(1)] If $S\subseteq U(R)$, then $P$ is $u$-$S$-pseudo-projective if and only if $P$ is pseudo-projective.
          \item[(2)] $u$-$S$-projective $\Rightarrow$  $u$-$S$-quasi-projective $\Rightarrow$  $u$-$S$-pseudo-projective. 
          \item[(3)] By (2) and \cite[ Proposition 3.6]{MM}, every $u$-$S$-semisimple module is $u$-$S$-pseudo-projective. 
      \end{enumerate}
      \begin{center}
    \end{center}
\end{remark}

For an $R$-module $M$, let $K\leq M$ denote that $K$ is a submodule of $M$. The following proposition gives some properties of $u$-$S$-pseudo-projective modules.

\begin{proposition}\label{prop1}
   Let $S$ be a multiplicative subset of a ring $R$. Then the following statements hold:
   \begin{enumerate}
   \item[(1)] Let $0\to A\xrightarrow{f} B\xrightarrow{g} C\to 0 $ be a $u$-$S$-split $u$-$S$-exact sequence. If $B$ is $u$-$S$-pseudo-projective, then so are $A$ and $C$.
    \item[(2)] If $A\oplus B$ is $u$-$S$-pseudo-projective, then so are $A$ and $B$. 
  \item[(3)] Let $f:A\to B$ be a $u$-$S$-isomorphism. Then $A$ is $u$-$S$-pseudo-projective if and only if $B$ is $u$-$S$-pseudo-projective. 
  \item[(4)] If $A$ is a $u$-$S$-pseudo-projective module, then any $u$-$S$-epimorphism $f:A\to A$ $u$-$S$-splits.
      
   \end{enumerate}  
\end{proposition}

\begin{proof} (1) Suppose that $0\to A\xrightarrow{f} B\xrightarrow{g} C\to 0 $ is a $u$-$S$-split $u$-$S$-exact sequence. Then there are $R$-homomorphisms $f':B\to A$ and $g':C\to B$ such that $f'f=t1_A$ and $gg'=t1_C$ for some $t\in S$. Let $K\leq A$ and $H=f'^{-1}(K)$. Then $H\leq B$. Let $\eta_K:A\to \frac{A}{K}$ and $\eta_H:B\to \frac{B}{H}$ be the natural maps. Since $B$ is $u$-$S$-pseudo-projective, there is $s\in S$ such that for any $u$-$S$-epimorphism $h:B\to \frac{B}{H}$, there is $e\in \text{End}_R(B)$ such that $sh=\eta_{H}e$. Let $h':A\to \frac{A}{K}$ be any $u$-$S$-epimorphism. Define $\overline{f}:\frac{A}{K}\to \frac{B}{H}$ by $\overline{f}(a+K)=f(a)+H$. If $a-b\in K$, then $f'(f(a)-f(b))=f'f(a-b)=t(a-b)\in K$ and so $f(a)-f(b)\in f'^{-1}(K)=H$. Thus $\overline{f}$ is well-defined. It is easy to check that $\overline{f}$ is an $R$-homomorphism. Now if $b\in B$, then $f'(b)\in A$ and so $tf'(b)=f'f(f'(b))$. It follows that $f'\big(tb-f(f'(b))\big)=0\in K$. So $tb-f(f'(b))\in f'^{-1}(K)=H$. Hence $t(b+H)=\overline{f}(f'(b)+K)\in \text{Im}(\overline{f})$. Thus $\overline{f}$ is a $u$-$S$-epimorphism. Let $h:=\overline{f}h'f':B\to \frac{B}{H}$. Then $h$ is a $u$-$S$-epimorphism since $f',h',\overline{f}$ are $u$-$S$-epimorphisms. So $sh=\eta_{H}e$ for some $e\in \text{End}_R(B)$. Let $e':=f'ef$ and $s':=st^2$. Then $e'\in \text{End}_R(A)$ and $st\overline{f}h'=s\overline{f}h't1_A=s\overline{f}h'f'f=shf=\eta_Hef$. Let $a\in A$ and let $h'(a)=a'+K$. Then $$stf(a')+H=st\overline{f}(a'+K)=st\overline{f}h'(a)=\eta_Hef(a)=ef(a)+H.$$ So $stf(a')-ef(a)\in H=f'^{-1}(K)$. It follows that $stf'f(a')-f'ef(a)\in K$. So $st^2a'-e'(a)\in K$. Hence $$s'h'(a)=st^2a'+K=e'(a)+K=\eta_Ke'(a).$$ Since $a$ was arbitrary, $s'h'=\eta_Ke'$. Thus $A$ is $u$-$S$-pseudo-projective. Similarly, we can show that $C$ is $u$-$S$-pseudo-projective.\\[0.1cm]       
(2) Let $i_A:A \to A\oplus B$ be the natural injection and $p_B: A\oplus B\to B$ be the natural projection. Since $0\to A\xrightarrow{i_A} A\oplus B\xrightarrow{p_B} B\to 0$ is a $u$-$S$-split $u$-$S$-exact sequence and $A\oplus B$ is $u$-$S$-pseudo-projective, then by part (1), $A$ and $B$ are $u$-$S$-pseudo-projective.  \\[0.1cm] 
(3) This follows from part (1) and the fact that the $u$-$S$-exact sequences $0\to 0\to A\xrightarrow{f} B\to 0$ and $0\to A\xrightarrow{f} B\to 0\to 0$ are $u$-$S$-split.\\[0.1cm] 
(4) Suppose that $A$ is a $u$-$S$-pseudo-projective module. Let $f:A\to A$ be any $u$-$S$-epimorphism and let $K=\text{Ker}(f)$. Then $u:\frac{A}{K}\to A$ given by $u(a+K)=f(a)$, $a\in A$, is a $u$-$S$-isomorphism. By \cite[Lemma 2.1]{ZQ}, there is a $u$-$S$-isomorphism $v:A\to \frac{A}{K}$ and $t\in S$ such that $uv=t1_A$. Since $K\leq A$ and $A$ is $u$-$S$-pseudo-projective, then there is an endomorphism $e:A\to A$ such that $sv=\eta_Ke$ for some $s\in S$. Thus $st1_A=suv=usv=u\eta_Ke=fe$. Therefore, $f$ is $u$-$S$-split. 
\end{proof}

Let $S$ be a multiplicative subset of a ring $R$ and $M$ an $R$-module. Recall that a submodule $N$ of $M$ is called a $u$-$S$-direct summand of $M$ if $M$ is $u$-$S$-isomorphic to $N\oplus N'$ for some $R$-module $N'$ \cite{KMOZ}.

\begin{proposition}\label{pro1}
      Let $S$ be a multiplicative subset of a ring $R$. If $P$ is $u$-$S$-pseudo-projective and $K\leq P$ such that $\frac{P}{K}$ is isomorphic to a direct summand of $P$, then $K$ is a $u$-$S$-direct summand of $P$.
\end{proposition}

\begin{proof}
 Let $A$ be a direct summand of $P$ such that $\frac{P}{K}\cong A$. Let $f:A\to \frac{P}{K}$ be an isomorphism. If $p_A:P\to A$ is the projection map and $\eta_K:P\to \frac{P}{K}$ is the natural map, then $P\xrightarrow{p_A} A\xrightarrow{f}\frac{P}{K}$ is a $u$-$S$-epimorphism. Now since $P$ is $u$-$S$-pseudo-projective, there is $e\in \text{End}_R(P)$ such that the following diagram 
 \[\xymatrix{
& P \ar[d]^{\eta_K}\\
 P \ar@{>}[ru]^-{ e}\ar[r]_{sfp_A}&\frac{P}{K}
}\] commutes for some $s\in S$. So $sfp_A=\eta_Ke$. Let $q:=ei_Af^{-1}:\frac{P}{K}\to P$, where $i_A:A\to P$ is the natural injection. Then $$\eta_Kq=\eta_Kei_Af^{-1}=sfp_Ai_Af^{-1}=sf1_Af^{-1}=sff^{-1}=s1_{\frac{P}{K}}.$$ So by \cite[Lemma 2.4]{ZQ}, the exact sequence $0\to K\to P\to \frac{P}{K}\to 0$ is $u$-$S$-split. Hence by \cite[Lemma 2.8]{KMOZ}, $P$ is $u$-$S$-isomorphic to $K\oplus \frac{P}{K}$. Therefore, $K$ is a $u$-$S$-direct summand of $P$.    
\end{proof}

\begin{corollary}\label{coro1}
       Let $S$ be a multiplicative subset of a ring $R$. Let $P$ be $u$-$S$-pseudo-projective and $e\in \text{End}_R(P)$. If $\text{Im}(e)$ is a direct summand of $P$, then $\text{Ker}(e)$ is a $u$-$S$-direct summand of $P$.
\end{corollary}

\begin{proof}
    Let $P$ be $u$-$S$-pseudo-projective and $e\in \text{End}_R(P)$. Then $\frac{P}{\text{Ker}(e)}\cong \text{Im}(e)$ is a direct summand of $P$. By Proposition \ref{pro1}, $\text{Ker}(e)$ is a $u$-$S$-direct summand of $P$.
\end{proof}
Let $S$ be a multiplicative subset of a ring $R$. Recall that $R$ is called $S$-von Neumann regular if for any $a \in R$, there exist $s \in S$ and $r \in R$ such that $sa =ra^2$ \cite{Z}. Equivalently, if the ring $R_S$ is von Neumann regular \cite[Proposition 3.10]{Z}. Next, we extend the notion of $S$-von Neumann regular rings to noncommutative rings. 
\begin{definition}
    Let $R$ be a noncommutative ring and $S$ a multiplicative subset of $R$. $R$ is said to be $S$-von Neumann regular if for any $a \in R$, there exist $s \in S$ and $r \in R$ such that $sa =ara$. 
\end{definition}  

\begin{thm}
Let $S$ be a multiplicative subset of a ring $R$ and let $P$ be a $u$-$S$-pseudo-projective faithful module such that $\text{Im}(e)$ is a direct summand of $P$ for every $e\in \text{End}_R(P)$. Then $\text{End}_R(P)$ is $\phi(S)$-von Neumann regular ring, where $\phi$ is the embedding of $R$ into $\text{End}_R(P)$. 
\end{thm}
   
\begin{proof}
    Let $e\in \text{End}_R(P)$. Then $\text{Im}(e)$ is a direct summand of $P$. By the proof of Proposition \ref{pro1}, the exact sequence $0\to \text{Ker}(e)\to P\xrightarrow{\eta} \frac{P}{\text{Ker}(e)}\to 0$ is $u$-$S$-split. So there is $s\in S$ and $R$-homomorphism $v:\frac{P}{\text{Ker}(e)}\to P$ such that $\eta v=s1_{\frac{P}{\text{Ker}(e)}}$. For $x\in P$, let $\overline{x}=x+\text{Ker}(e)$. Then $s\overline{x}=\eta v(\overline{x})=\overline{v(\overline{x})}$. So $sx-v(\overline{x})\in \text{Ker}(e)$, It follows that $sP\subseteq \text{Ker}(e)+\text{Im}(v)$. Also, we have $s\big( \text{Ker}(e)\cap\text{Im}(v)\big)=0$. Indeed, if $v(\overline{y})\in \text{Ker}(e)$, then $s\overline{y}=\eta v(\overline{y})=\overline{v(\overline{y})}=\overline{0}$ and so $sv(\overline{y})=0$. Let $P=A\oplus \text{Im}(e)$ be an internal direct sum. We have $se(P)\subseteq e(\text{Im}(v))$. Define $q:P\to P$ by $q(a+e(x))=sb$ where $a\in A$, $x\in P$, and $b\in \text{Im}(v)$ is such that $se(x)=e(b)$. Then $q$ is well defined. Indeed, if $a+e(x)=a'+e(x')$, then $a-a'=e(x'-x)\in A\cap \text{Im}(e)=0$. So $a=a'$ and $x'-x\in \text{Ker}(e)$. Write $se(x)=e(b)$ and $se(x')=e(b')$ for some $b,b'\in \text{Im}(v)$. Since $e(b'-b)=se(x'-x)=0$, then $b'-b\in \text{Ker}(e)\cap\text{Im}(v)$ and so $s(b'-b)=0$. Hence $sb=sb'$. Finally, we show that $s^2e=eqe$. Let $x\in P$. So $se(x)=e(b)$ for some $b\in \text{Im}(v)$. Thus $eqe(x)=e(q(0+e(x)))=e(sb)=se(b)=s^2e(x)$. Since $x\in P$ was an arbitrary, we have $s^2e=eqe$. Let $t=\phi(s^2)$. Then for any $x\in P$, $te(x)=\phi(s^2)e(x)=s^2e(x)=eqe(x)$. Thus $te=eqe$. Therefore, $\text{End}_R(P)$ is $\phi(S)$-von Neumann regular ring.  
\end{proof}

\begin{lemma}\label{lem1}
  Let $S$ be a multiplicative subset of a ring $R$. Let $P$ be a $u$-$S$-pseudo-projective and $A$ a direct summand of $P$. Then for any $K\leq A$, there is $s\in S$ such that for any $u$-$S$-epimorphism $g:P\to \frac{A}{K}$, $sg$ can be lifted to a homomorphism $h:P\to A$.   
\end{lemma}

\begin{proof}
  Let $P=A\oplus B$, $K\leq A$, and $H=K\oplus 0$. Identify $\frac{P}{H}=\frac{A\oplus B}{K\oplus 0}$ with $\frac{A}{K}\oplus B$. Since $P$ is $u$-$S$-pseudo-projective, there is $s\in S$ such that for any $u$-$S$-epimorphism $f:P\to \frac{P}{H}$, $sf$ can be lifted to an endomorphism $e:P\to P$. Let $g:P\to \frac{A}{K}$ be any $u$-$S$-epimorphism. Then the $R$-homomorphism $f:P\to \frac{P}{H}$ given by $f(a,b)=(g(a,b),b)$ is a $u$-$S$-epimorphism. So there is $e\in \text{End}_R(P)$ such that the following diagram 
 \[\xymatrix{
& P \ar[d]^{\eta_H}\\
 P \ar@{>}[ru]^-{ e}\ar[r]_{sf}&\frac{P}{H}
}\] commutes. That is, $sf=\eta_He$. Let $p_A:P\to A$, $p_{\frac{A}{K}}:\frac{A}{K}\oplus B\to \frac{A}{K}$ be the natural projections, and let $h:=p_Ae:P\to A$. We claim that the following diagram 
 \[\xymatrix{
& A \ar[d]^{\eta_K}\\
 P \ar@{>}[ru]^-{ h}\ar[r]_{sg}&\frac{A}{K}
}\] commutes. For $(a,b)\in P$, let $e(a,b)=(a',b')$. Then $\eta_K h(a,b)=\eta_Kp_Ae(a,b)=\eta_K(a')=a'+K=p_{\frac{A}{K}}(a'+K,b')=p_{\frac{A}{K}}\eta_H e(a,b)=p_{\frac{A}{K}}sf(a,b)=sg(a,b)$. Thus $sg=\eta_Kh$.  
\end{proof}

\begin{proposition}\label{prop2}
     Let $S$ be a multiplicative subset of a ring $R$ and $A\oplus B$ be a $u$-$S$-pseudo-projective module. Then 
     
  \begin{enumerate}
      \item[(1)] $A$ is $u$-$S$-projective relative to $B$ and $B$ is $u$-$S$-projective relative to $A$.
      \item[(2)] any $u$-$S$-epimorphism $f:A\to B$ $u$-$S$-splits.
  \end{enumerate}   
\end{proposition}

\begin{proof} (1) Suppose that $A\oplus B$ is $u$-$S$-pseudo-projective. Let $K\leq B$. Then by Lemma \ref{lem1}, there is $s\in S$ such that for any $u$-$S$-epimorphism $g:A\oplus B\to \frac{B}{K}$, $sg$ can be lifted to a homomorphism $h:A\oplus B\to B$. Let $f:A\to \frac{B}{K}$ be any homomorphism and $\eta_K:B\to \frac{B}{K}$ be the natural map. Define $g:A\oplus B\to \frac{B}{K}$ by $g(a,b)=f(a)+\eta_K(b)$. Then $g$ is an epimorphism and so it is a $u$-$S$-epimorphism. Hence $sg=\eta_Kh$ for some homomorphism $h:A\oplus B\to B$. Let $i_A:A\to A\oplus B$ be the natural injection and $h':=hi_A:A\to B$. Then for $a\in A$, $$\eta_Kh'(a)=\eta_Khi_A(a)=sg(a,0)=s(f(a)+\eta_K(0))=sf(a).$$ So $sf=\eta_Kh'$. Hence, the map $(\eta_K)_{*}:\text{Hom}_R(A,B)\to \text{Hom}_R(A,\frac{B}{K})$ is a $u$-$S$-epimorphism. By \cite[Theorem 2.3]{MM}, $A$ is $u$-$S$-projective relative to $B$. Similarly, we can show that $B$ is $u$-$S$-projective relative to $A$.\\[0.1cm]
(2) Let $f:A\to B$ be any $u$-$S$-epimorphism. Since $A\oplus B$ is $u$-$S$-pseudo-projective, then by part (1), $B$ is $u$-$S$-projective relative to $A$ and hence by \cite[Lemma 2.15]{MM}, the $u$-$S$-exact sequence $0\to \text{Ker}(f)\to A\xrightarrow{f} B\to 0$ $u$-$S$-splits. Thus $f$ $u$-$S$-splits.
\end{proof}

\begin{corollary}\label{cor1}
     Let $S$ be a multiplicative subset of a ring $R$ and $M$ an $R$-module. Then $M$ is $u$-$S$-quasi-projective if and only if $M\oplus M$ is $u$-$S$-pseudo-projective. 
     \end{corollary}

\begin{proof}
   Suppose that $M$ is $u$-$S$-quasi-projective. Then $M\oplus M$ is $u$-$S$-quasi-projective by \cite[Proposition 3.8 ]{MM}. So $M\oplus M$ is $u$-$S$-pseudo-projective by Remark \ref{rem1} (2). The converse follows from Proposition \ref{prop2} (1).   
\end{proof}

Let $M$ be an $R$-module. For a positive integer $n$, let $M^{(n)}=\underbrace{M\oplus M\oplus\cdots \oplus M}_{n\text{-times}}$.

\begin{corollary}\label{cor2}
  Let $S$ be a multiplicative subset of a ring $R$ and $M$ an $R$-module. For any integer $n\geq 2$, $M$ is $u$-$S$-quasi-projective if and only if $M^{(n)}$ is $u$-$S$-pseudo-projective. 
\end{corollary}

\begin{proof}
 ($\Rightarrow$). Since $M$ is $u$-$S$-quasi-projective, $M$ is $u$-$S$-projective relative to $M$. Hence by \cite[Proposition 3.8]{MM}, $M^{(n)}$ is $u$-$S$-quasi-projective and so by Remark \ref{rem1} (2), $M^{(n)}$ is $u$-$S$-pseudo-projective. \\
  ($\Leftarrow$). For $n=2$, apply Corollary \ref{cor1}. For $n>2$, since $M^{(2)}\oplus M^{(n-2)}\cong  M^{(n)}$ is $u$-$S$-pseudo-projective, then by Proposition \ref{prop1} (2), $M^{(2)}$ is $u$-$S$-pseudo-projective. Thus $M$ is $u$-$S$-quasi-projective by Corollary \ref{cor1}. 
\end{proof}

The following example gives a $u$-$S$-pseudo-projective module that is not $u$-$S$-projective.
\begin{example} \label{ex1}
    First, let $\mathbb{Z}^+$ ($\mathbb{P}$) denote the set of all positive integers (prime numbers). Let $R=\mathbb{Z}$, $S=\mathbb{Z}^+$, and $M:=\bigoplus\limits_{p\in\mathbb{P}}\mathbb{Z}_p$. Then $M$ is a $u$-$S$-quasi-projective module that is not $u$-$S$-projective by \cite[Example 3.3]{MM}. By Corollary \ref{cor1}, $M\oplus M$ is a $u$-$S$-pseudo-projective module. However, $M\oplus M$ is not $u$-$S$-projective by \cite[ Corollary 2.8 (1)]{MM}.
\end{example}

\begin{proposition}\label{propn1}
     Let $S$ be a multiplicative subset of a ring $R$, $M$ an $R$-module, and $f:P\to M$ an epimorphism with $P$ projective. Then $M$ is $u$-$S$-projective if and only if $P\oplus M$ is $u$-$S$-pseudo-projective.
\end{proposition}

\begin{proof} Suppose that $M$ is $u$-$S$-projective. Since $P$ is projective, then by \cite[ Corollary 2.11]{ZQ}, $P$ is $u$-$S$-projective. So by \cite[Proposition 2.14 (1)]{ZQ}, $P\oplus M$ is $u$-$S$-projective and hence by Remark \ref{rem1} (2), $P\oplus M$ is $u$-$S$-pseudo-projective. Conversely, suppose that $P\oplus M$ is $u$-$S$-pseudo-projective. Then by Proposition \ref{prop2} (2) and since $f:P\to M$ is a $u$-$S$-epimorphism, we have $f$ $u$-$S$-splits. So the exact sequence $0\to \text{Ker}(f)\to P\xrightarrow{f} M\to 0$ $u$-$S$-splits. Hence by \cite[Lemma 2.8]{KMOZ}, $P$ is $u$-$S$-isomorphic to $\text{Ker}(f)\oplus M$. Since $P$ is $u$-$S$-projective, then by \cite[Proposition 2.14 (3)]{ZQ} and \cite[Corollary 2.8 (1)]{MM}, $M$ is $u$-$S$-projective. 
\end{proof}

\begin{corollary}\label{cor3}
    Let $R$ be a ring, $M$ an $R$-module, and $f:P\to M$ an epimorphism with $P$ projective. Then $M$ is projective if and only if $P\oplus M$ is pseudo-projective.
\end{corollary}

For an $R$-module $M$, let $E(M)$ denote the injective envelope of $M$.

\begin{proposition}
   Let $S$ be a multiplicative subset of a ring $R$ and $M$ an $R$-module. Then $M$ is $u$-$S$-injective if and only if $M\oplus E(M)$ is $u$-$S$-pseudo-injective. 
\end{proposition}

\begin{proof}
   Let $M$ be $u$-$S$-injective. Then $M\oplus E(M)$ is $u$-$S$-injective by \cite[ Proposition 4.7]{QK}. So by \cite[Remark 2.4 (2)]{MM2}, $M\oplus E(M)$ is $u$-$S$-pseudo-injective. Conversely, suppose that $M\oplus E(M)$ is $u$-$S$-pseudo-injective. Then by \cite[ Theorem 2.6]{MM2}, the monomorphism $i_M:M\to E(M)$ $u$-$S$-splits. So the exact sequence $0\to M\to E(M)\to \frac{E(M)}{M}\to 0$ $u$-$S$-splits. Hence by \cite[Lemma 2.8]{KMOZ}, $E(M)$ is $u$-$S$-isomorphic to $M\oplus \frac{E(M)}{M}$. But $E(M)$ is $u$-$S$-injective, so by \cite[ Proposition 4.7 (3)]{QK} and \cite[Corollary 2.8 (2)]{MM}, $M$ is $u$-$S$-injective.
\end{proof}

\begin{corollary}
  Let $R$ be a ring and $M$ an $R$-module. Then $M$ is injective if and only if $M\oplus E(M)$ is pseudo-injective. 
\end{corollary}

\section{Characterizations of some classes of rings}
In this section, we characterize certain classes of rings, such as $u$-$S$-semisimple and strongly $S$-perfect rings. We start with the following theorem that gives several characterizations of $u$-$S$-semisimple rings.

\begin{thm}\label{thm1}
Let $S$ be a multiplicative subset of a ring $R$. Then the following statements are equivalent: 
  \begin{enumerate}
          \item[(1)] $R$ is $u$-$S$-semisimple;
          \item[(2)] every $R$-module is $u$-$S$-pseudo-projective;
          \item[(3)] every direct sum of an $R$-module and a free $R$-module is $u$-$S$-pseudo-projective ($u$-$S$-pseudo-injective);
         \item[(4)] every direct sum of an $R$-module and a free $R$-module is $u$-$S$-quasi-projective ($u$-$S$-quasi-injective).
      \end{enumerate}
  \end{thm}
\begin{proof}
$(1)\Rightarrow (2)$. Let $R$ be a $u$-$S$-semisimple ring. Then every $R$-module is $u$-$S$-quasi-projective by \cite[Theorem 3.11]{MM}. So by Remark \ref{rem1} (2), every $R$-module is $u$-$S$-pseudo-projective. \\
$(2)\Rightarrow (1)$. Let $M$ be any $R$-module and let $f:P\to M$ be an epimorphism with $P$ projective. Then by (2), $P\oplus M$ is $u$-$S$-pseudo-projective. So by Proposition \ref{propn1}, $M$ is $u$-$S$-projective. Hence every $R$-module is $u$-$S$-projective. Therefore, $R$ is $u$-$S$-semisimple by \cite[Theorem 3.5]{ZQ}.\\
 $(1)\Rightarrow (3)$. This follows from $(1)\Rightarrow (2)$ and \cite[Theorem 2.14]{MM2}.\\
 $(3)\Rightarrow (1)$. Let $F$ be any free $R$-module. Then by (3), for any $R$-module $M$, $M\oplus F$ is $u$-$S$-pseudo-projective ($u$-$S$-pseudo-injective). By Proposition \ref{prop2} (1) (\cite[Proposition 2.13]{MM2}), $M$ is $u$-$S$-projective ($u$-$S$-injective) relative to $F$ for any $R$-module $M$. Thus every $R$-module is $u$-$S$-projective ($u$-$S$-injective) relative to $F$, By \cite[Theorem 2.16]{MM}, $F$ is $u$-$S$-semisimple. Hence, every free $R$-module is $u$-$S$-semisimple. Therefore, $R$ is $u$-$S$-semisimple.\\
$(1)\Leftrightarrow (4)$. This follows from \cite[Theorem 2.16, Proposition 3.8, and Theorem 3.11]{MM}.
\end{proof}

Recall that an $R$-module $F$ is called $u$-$S$-flat if the induced sequence $0\to A \otimes_R F\to  B \otimes_R F \to C\otimes_R F \to 0$ is $u$-$S$-exact for any $u$-$S$-exact sequence $0 \to A \to B \to C \to 0$ \cite{Z}, and that a ring $R$ is called a strongly $S$-perfect ring if every $u$-$S$-flat is projective \cite{AZ}. The following theorem characterizes strongly $S$-perfect rings.

\begin{thm}\label{thmp}
    Let $S$ be a multiplicative subset of a ring $R$. Then the following statements are equivalent: 
  \begin{enumerate}
          \item[(1)] $R$ is strongly $S$-perfect;
          \item[(2)] every $u$-$S$-flat is quasi-projective;
          \item[(3)] every $u$-$S$-flat is pseudo-projective.
      \end{enumerate}
  \end{thm}

\begin{proof}
    $(1)\Rightarrow (2)\Rightarrow (3)$. Clear. \\
$(3)\Rightarrow (1)$. Let $F$ be any $u$-$S$-flat $R$-module and $f:P\to F$ be an epimorphism with $P$ projective. Since $P$ is projective, it is $u$-$S$-flat by \cite[Corollary 2.11 and Proposition 2.13]{ZQ}. So by \cite[Proposition 3.4 (2)]{Z}, $P\oplus F$ is $u$-$S$-flat. By (3), $P\oplus F$ is pseudo-projective, and so by Corollary \ref{cor3}, $F$ is projective. Hence, every $u$-$S$-flat $R$-module is projective. Therefore, $R$ is a strongly $S$-perfect ring.
\end{proof}

The following theorem characterizes rings in which every $u$-$S$-quasi-projective module is $u$-$S$-projective.

\begin{thm}\label{thm2}
   Let $S$ be a multiplicative subset of a ring $R$. Then the following statements are equivalent:
   \begin{enumerate}
          \item[(1)] Every $u$-$S$-quasi-projective module is $u$-$S$-projective;
          \item[(2)] Every direct sum of two $u$-$S$-quasi-projective modules is $u$-$S$-projective;
          \item[(3)] Every direct sum of two $u$-$S$-quasi-projective modules is $u$-$S$-quasi-projective.
      \end{enumerate}
  \end{thm}

\begin{proof}
   $(1)\Rightarrow (2)$. Let $M$ and $N$ be two $u$-$S$-quasi-projective modules. Then by (1), $M$ and $N$ are $u$-$S$-projective modules. So $M\oplus N$ is $u$-$S$-projective by \cite[Proposition 2.14 (1)]{ZQ}. \\
   $(2)\Rightarrow (3)$. This is clear from Remark \ref{rem1} (2). \\
   $(3)\Rightarrow (1)$. Let $M$ be a $u$-$S$-quasi-projective module. Let $f:P\to M$ be an epimorphism with $P$ projective. Then by (3), $P\oplus M$ is $u$-$S$-quasi-projective. Let $p_1:P\oplus M\to P$ and $p_2:P\oplus M\to M$ be the natural projections, and let $i_2:M\to P\oplus M$ be the natural injection. Since $P\oplus M$ is $u$-$S$-quasi-projective and $fp_1:P\oplus M\to M$ is an epimorphism, then by \cite[ Theorem 2.3]{MM}, there is $e\in \text{End}_R(P\oplus M)$ such that the following diagram 
 \[\xymatrix{
& P\oplus M \ar[d]^{sp_2}\\
 P\oplus M \ar@{<-}[ru]^-{ e}\ar[r]_{fp_1}&M
}\] commutes for some $s\in S$. So $sp_2=fp_1e$. Since $p_2i_2=1_M$, then $s1_M=sp_2i_2=fp_1ei_2$. Let $q:=p_1ei_2:M\to P$. Then $s1_M=fq$. So by \cite[Lemma 2.4]{ZQ}, the exact sequence $0\to \text{Ker}(f)\to P\xrightarrow{f} M\to 0$ $u$-$S$-splits. By \cite[Lemma 2.8]{KMOZ}, $P$ is $u$-$S$-isomorphic to $\text{Ker}(f)\oplus M$. Since $P$ is $u$-$S$-projective, $\text{Ker}(f)\oplus M$ is $u$-$S$-projective by \cite[ Proposition 2.14 (3)]{ZQ}. Hence $M$ is $u$-$S$-projective by \cite[Corollary 2.8 (1)]{MM}. Thus (1) holds. 
\end{proof}

The following theorem characterizes rings in which every $u$-$S$-pseudo-projective module is $u$-$S$-projective.
    
\begin{thm}\label{thm3} 
  Let $S$ be a multiplicative subset of a ring $R$. Then the following statements are equivalent:
    \begin{enumerate}
          \item[(1)] Every $u$-$S$-pseudo-projective module is $u$-$S$-projective.
          \item[(2)] Every direct sum of two $u$-$S$-pseudo-projective modules is $u$-$S$-projective.
          \item[(3)] Every direct sum of two $u$-$S$-pseudo-projective modules is $u$-$S$-pseudo-projective. 
      \end{enumerate}
\end{thm}

\begin{proof}
$(1)\Rightarrow (2)$. Let $M$ and $N$ be two $u$-$S$-pseudo-projective modules. Then $M$ and $N$ are $u$-$S$-projective modules and so $M\oplus N$ is $u$-$S$-projective. \\
$(2)\Rightarrow (3)$. This follows from Remark \ref{rem1} (2). \\
$(3)\Rightarrow (1)$. Let $M$ be any $u$-$S$-pseudo-projective module. Let $f:P\to M$ be an epimorphism with $P$ projective. Then by (3), $P\oplus M$ is $u$-$S$-pseudo-projective. Thus $M$ is $u$-$S$-projective by Proposition \ref{propn1}.
\end{proof}

An application of Theorem \ref{thm3} is the following example, which shows that a direct sum of two $u$-$S$-pseudo-projective modules need not be $u$-$S$-pseudo-projective. 

\begin{example}\label{ex2}
Let $R=\mathbb{Z}$ and $S=\mathbb{Z}^+$. By Example \ref{ex1}, there is an $R$-module $U$ that is $u$-$S$-pseudo-projective but not $u$-$S$-projective. So by Theorem \ref{thm3}, there are two $u$-$S$-pseudo-projective $R$-modules $A$ and $B$ such that $A\oplus B$ is not $u$-$S$-pseudo-projective.  

\end{example}

\textbf{Conflict of interest:} The authors declare that they have no conflict of interest.\\[0.2cm]


\begin{thebibliography}{99}

\normalsize
\baselineskip=17pt

\bibitem{MM} M. Adarbeh and M. Saleh, Uniformly $S$-projective relative to a module and its dual, https://doi.org/10.48550/arXiv.2509.01646
\bibitem{MM2} M. Adarbeh and M. Saleh, Uniformly $S$-pseudo-injective modules, Ann Univ Ferrara, 2026, to appear.
\bibitem{AF} F. W. Anderson and K. R. Fuller, Rings and Categories of Modules (Springer-Verlag 1974).
\bibitem{AZ} R. A. K. Assaad and X. Zhang, $S$-cotorsion modules and dimensions, Hacet.
J. Math. Stat. (2022) 1-10.
\bibitem{KMOZ} H. Kim, N. Mahdou, E. H. Oubouhou, and X. Zhang, Uniformly $S$-projective modules and uniformly $S$-projective uniformly $S$-covers, Kyungpook Math. J., 64 (2024), 607–618.
  \bibitem{QK} M. Z. Chen, H. Kim, W. Qi, F. G. Wang, W. Zhao, Uniformly $S$-Noetherian rings. Quaest. Math. 47(5), 1019–1038 (2024).
\bibitem{SS} S. Singh and A. K. Srivastava, Dual automorphism-invariant modules, J. Algebra 371 (2012), 262-275.
\bibitem{WK} F. Wang and H. Kim, Foundations of Commutative Rings and Their Modules, Algebra and
 Applications, vol. 22, Springer, Singapore, 2016.
\bibitem{ZQ} X. Zhang and W. Qi, Characterizing $S$-projective modules and $S$-semisimple rings by
 uniformity, J. Commut. Algebra, 15(1) (2023), 139-149. 
\bibitem{Z} X. Zhang, Characterizing $S$-flat modules and $S$-von Neumann regular rings by uniformity,
 Bull. Korean Math. Soc., 59(3) (2022), 643-657.


\end{thebibliography}
\end{document}